\documentclass[12pt, twoside]{article}
\usepackage{amsmath,amsthm,amssymb}
\usepackage{times}
\usepackage{enumerate}
\usepackage{bm}
\usepackage[all]{xy}
\usepackage{mathrsfs}
\usepackage{amscd}
\usepackage{color}
\def\titlerunning#1{\gdef\titrun{#1}}
\makeatletter
\def\author#1{\gdef\autrun{\def\and{\unskip, }#1}\gdef\@author{#1}}
\def\address#1{{\def\and{\\\hspace*{18pt}}\renewcommand{\thefootnote}{}%
		\footnote {#1}}%
	\markboth{\autrun}{\titrun}}
\makeatother
\def\email#1{e-mail: #1}

\def\keywords#1{\par\medskip
	\noindent\textbf{Keywords.} #1}

\newtheorem{theorem}{Theorem}[section]
\newtheorem{corollary}[theorem]{Corollary}
\newtheorem{lemma}[theorem]{Lemma}
\newtheorem{proposition}[theorem]{Proposition}
\theoremstyle{definition}
\newtheorem{definition}[theorem]{Definition}
\newtheorem{remark}[theorem]{Remark}
\newtheorem{example}[theorem]{Example}
\numberwithin{equation}{section}

\newtheorem{conjecture}{Conjecture}

\xyoption{all}

\def \C {\mathbb{C}}

\def \a {\alpha }
\def \b {\beta}

\def \de {\delta}
\def \De {\Delta}
\def \la {\lambda}
\def \La {\Lambda}
\def\w {\omega}
\def\Om{\Omega}

\def\pa{\partial}
\def\na {\nabla}
\def\Ga{\Gamma}

\begin{document}
	
\baselineskip=17pt

\titlerunning{On Euler number of symplectic hyperbolic manifold}
\title{On Euler number of symplectic hyperbolic manifold}

\author{Teng Huang}

\date{}

\maketitle

\address{Teng Huang: School of Mathematical Sciences, University of Science and Technology of China ; CAS Key Laboratory of Wu Wen-Tsun Mathematics,  University of Science and Technology of China, Hefei, Anhui, 230026, People’s Republic of China; \email{htmath@ustc.edu.cn;htustc@gmail.com}}
\begin{abstract}
In this article, we introduce a class of closed $2n$-dimensional almost K\"{a}hler manifold $X$ which called the special symplectic hyperbolic manifold. Those manifolds include K\"{a}hler hyperbolic manifolds. We study the spaces of $L^{2}$-harmonic forms on the universal covering space of $X$. We then prove the Singer conjecture on special symplectic hyperbolic case. As an application, we can show that the Euler number of a special symplectic manifold satisfies the inequality $(-1)^{n}\chi(X)>0$. 
\end{abstract}
\keywords{Hopf conjecture, Singer conjecture, almost K\"{a}hler manifold, Euler number}
\section{Introduction}
Let us start the article by recalling two well-known conjectures related to the negativity of Riemannian sectional curvature. The first one, usually attributed to Hopf, is
\begin{conjecture}(Hopf Conjecture)
Let $X$ be a closed 2n-dimensional Riemannian manifold with sectional curvature $sec$. Then
\begin{equation*}
\left\{
\begin{aligned}
(-1)^{n}\chi(X)>0, &\ if\ sec<0\\
(-1)^{n}\chi(X)\geq0,&\ if\ sec\leq0.\\
\end{aligned}  
\right.
\end{equation*}	
\end{conjecture}
This is true for $n=1$ and $2$ as the Gauss–Bonnet integrands in these two low dimensional cases have the desired sign \cite{Chern}. However, in higher dimensions, it is known that the sign of the sectional curvature does not determine the sign of the Gauss-Bonnet-Chern integrand \cite{Ger}. A vanishing theorem in \cite{JX} which stated that the space of $L^{2}$ $k$-forms is trivial for $k$ in a certain range which depends on pinching constants for the curvature. For good pinching constants the question of Hopf can thus be answered.  

We denote by $h^{k}_{(2)}(X)$ the $k$-th $L^{2}$-Betti number of Riemannian manifold $X$. The second conjecture which proposed by Singer (\cite{Sin} and also \cite[Conjecture 2]{Dod}) is 
\begin{conjecture}(Singer Conjecture)
Let $X$ be a closed 2n-dimensional Riemannian manifold with negative sectional curvature. Then
\begin{equation*}
\left\{
\begin{aligned}
h^{k}_{(2)}(X)=0, & k\neq n\\
h^{n}_{(2)}(X)>0.& \\
\end{aligned}  
\right.
\end{equation*}	
\end{conjecture}
Because of the Euler-Poincar\'{e} formula $$\chi(X)=\sum_{k\geq0}(-1)^{k}h_{(2)}^{k}(X)$$
the Singer conjecture implies the Hopf conjecture in the case where $X$  has negative sectional curvature.

The program outlined above was carried out by Gromov \cite{Gro} when the manifold in question is K\"{a}hler and is homotopy equivalent to a closed manifold with strictly negative sectional curvatures.

A differential form $\a$ in a Riemannian manifold $(X,g)$ is called bounded with respect to the metric $g$ if the $L^{\infty}$-norm of $\a$ is finite, namely,
$$\|\a\|_{L^{\infty}(X)}=\sup_{x\in X}|\a(x)|<\infty.$$
By definition, a $k$-form $\a$ is said to be $d$(bounded) if $\a=d\b$, where $\b$ is a bounded $(k-1)$-form. It is obvious that if $X$ is compact, then every exact form is $d$(bounded). However, when $X$ is not compact, there exist smooth differential forms which are exact but not $d$(bounded). For instance, on $\mathbb{R}^{n}$, $\a=dx^{1}\wedge\cdots\wedge dx^{n}$ is exact, but it is not $d$(bounded).  Let’s recall some concepts introduced in \cite{CX,JZ}. A differential form $\a$ on a complete non-compact Riemannian manifold $(X,g)$ is called $d$(sublinear) if there exist a differential form $\b$ and a number $c>0$ such that $\a=d\b$ and 
$$\ |\b(x)|_{g}\leq c(1+\rho_{g}(x,x_{0})),$$  
where $\rho_{g}(x,x_{0})$ stands for the Riemannian distance between $x$ and a base point $x_{0}$ with respect to $g$.

Let $(X, g)$ be a closed Riemannian manifold and $\pi:(\tilde{X},\tilde{g})\rightarrow(X,g)$ be the universal covering with $\tilde{g}=\pi^{\ast}g$. A form $\a$ on $X$ is called $\tilde{d}$(bounded) (resp. $\tilde{d}$(sublinear)) if $\pi^{\ast}\a$ is a $d$(bounded) (resp. $d$(sublinear)) form on $(\tilde{X},\tilde{g})$. In geometry, various notions of hyperbolicity have been introduced, and the typical examples are manifolds with negative curvature in suitable sense \cite{CY}. The starting point for the present investigation is Gromov's notion of K\"{a}hler hyperbolicity \cite{Gro}. Gromov \cite{Gro} pointed out that if the Riemannian manifold $(X,g)$ is a complete simply-connected manifold and it has strictly negative sectional curvatures, then every smooth bounded closed form of degree $k\geq2$ is $d$(bounded). Then he proved the Hopf conjecture in the K\"{a}hler case. Gromov \cite{Gro} also gave a lower bound on the spectra of the Laplace operator $\De_{d}:=dd^{\ast}+d^{\ast}d$ on $L^{2}$-forms in $\Om^{p,q}(X)$ for $p+q\neq n$ to sharpen the Lefschetz vanishing theorem. In order to attack Hopf conjecture in the K\"{a}hler manifold with $sec\leq0$, by extending Gromov’s idea, Cao-Xavier \cite{CX} and Jost-Zuo  \cite{JZ}  independently introduced the concept of K\"{a}hler parabolicity,  which includes nonpositively curved closed K\"{a}hler manifolds, and showed that their Euler numbers have the desired property. In \cite{Hua2}, the author proved the Hopf conjecture in some locally conformally K\"{a}hler manifolds case. 

Let $(X,J,\w)$ be a closed $2n$-dimensional symplectic manifold. Let $J$ be an $\w$-compatible almost complex structure, i.e., $J^{2}=-id$, $\w(J\cdot, J\cdot)=\w(\cdot,\cdot)$,  and $g(\cdot,\cdot)=\w(\cdot, J\cdot)$ is a Riemannian metric on $X$. The triple $(\w, J, g)$ is called an almost K\"{a}hler structure on $X$. Notice that any one of the pairs $(\w, J)$, $(J, g)$ or $(g, \w)$ determines the other two. An almost-K\"{a}hler metric $g$ is K\"{a}hler if and only if $J$ is integrable. For symplectic case, inspired by K\"{a}hler geometry, Tan-Wang-Zhou \cite{TWZ} gave the definition of symplectic parabolic manifold.  A closed almost K\"{a}hler manifold $(X,J,\w)$ is called symplectic hyperbolic (resp. parabolic) if the lift $\tilde{\w}$ of $\w$ to the universal covering $(\tilde{X},\tilde{J},\tilde{\w})\rightarrow(X,J,\w)$ is $d$(bounded) (resp. $d$(sublinear)) on $(\tilde{X},\tilde{J},\tilde{\w})$.

Noticing that the proof of the vanishing theorem on K\"{a}hlerian case is based on the identity $[L,\De_{d}]=0$ due to the K\"{a}hler identities. But, in general, the complete almost K\"{a}hler manifold $(X,J,\w)$ is not K\"{a}hlerian. By considering Tseng-Yau’s new symplectic cohomologies on symplectic parabolic manifold, Tan-Wang-Zhou \cite{TWZ,HTan} proved that if $(X,J,\w)$ is a closed $2n$-dimensional symplectic parabolic manifold which satisfies the Hard Lefschetz Condition \cite{Mat,Yan}, then the Euler number of $X$ satisfies $(-1)^{n}\chi(X)\geq0$.  Hind-Tomassini \cite{HT} constructed a $d$(bounded) complete almost K\"{a}hler manifold $X$ satisfies $\mathcal{H}^{1}_{(2)}(X)\neq\{0\}$ by using methods of contact geometry. The Hard Lefschetz Condition is necessary in Tan-Wang-Zhou's theorem. Hence it is not enough to prove the Hopf conjecture in the Dodziuk-Singer's program using only the condition that $\w$ is $d$(bounded).

In this article, we observe that the operator $\De_{d'}=\frac{1}{4}(\De_{d}+\De_{d^{\La}})$ commutes to $L$, where the opearator $d'$ is defined in Section 3.1. We also observe that if $\dim\mathcal{H}^{k}_{d'}(X)=\dim\mathcal{H}^{k}_{d}(X)$ holds for any $0\leq k<n$ over a closed almost K\"{a}hler manifold $X$, then $X$ satisfies Hard Lefschetz Condition (see Corollary \ref{C1}). For non-compact case, we use Gromov's method to study the operator $d'+d'^{\ast}:\Om^{even}\rightarrow\Om^{odd}$ on a complete almost K\"{a}hler $2n$-manifold $X$ with $d$(bounded) symplectic $\w$. We then give a lower bound on the spectra of the  operator $\De_{d'}$ on the space of $L^{2}$-forms $\Om^{k}(X)$ for $k\neq n$ (see Theorem \ref{T4}).  In \cite{Gro}, Gromov proved that if $X$ is a closed K\"{a}hler hyperbolic manifold, then $(-1)^{n}\chi(X)>0$. Gromov's method amounted to construct arbitrarily small perturbations of $d+d^{\ast}$ on $\Om^{\ast}_{(2)}(X)$ with a non-trivial $L^{2}$-kernel. For this, he applied the $L^{2}$-index theorem to a twisted $d+d^{\ast}$, i.e., to vector valued differential forms. We observe that the index of $d'+d'^{\ast}:\Om^{even}\rightarrow\Om^{odd}$ is equal to the Euler number of $X$, that is, $Index(d'+d'^{\ast})=\chi(X)$ (see Corollary \ref{C2}). By the Atiyah's $L^{2}$ index, the $L^{2}$-index of $\widetilde{d'+d'^{\ast}}$ is equal to the index of $d'+d'^{\ast}$. We construct arbitrarily small perturbations of $\widetilde{d'+d'^{\ast}}$ on $\Om^{\ast}(\tilde{X})$ with a non trivial $L^{2}$-kernel. 

Let $(X,J,\w)$ be an almost K\"{a}hler manifold with a symplectic form $\w$. The symplectic form $\w$ is closed on $X$, this is equivalent to $g((\na_{X}J)Y,Z)+g((\na_{Z}J)X,Y)+g((\na_{Y}J)Z,X)=0$. We denote by $\tau^{\ast}$, $\tau$ the $\ast$-curvature and scalar curvature  of $X$, respectively.  There is a known identity $|\na J|^{2}=2(\tau^{\ast}-\tau)$ \cite{Sek,SV}. We introduce a class of closed almost K\"{a}hler manifolds as follows. 
\begin{definition}\label{D1}
A closed almost K\"{a}hler manifold $(X,J,\w)$ is called special symplectic hyperbolic if there is a bounded $1$-form $\theta$ on $\tilde{X}$ such that $\tilde{\w}=d\theta$ and 
\begin{equation}\label{E9}
\|\theta\|^{-2}_{L^{\infty}(\tilde{X})}\geq C\max_{x\in X}(\tau^{\ast}-\tau)(x),
\end{equation}
where $C$ is an uniformly positive constant only depend on $n$. 
\end{definition}
Following the identity $g(N_{J}(X,Y),Z)=2g(J(\na_{Z}J)X,Y)$, the Equation  (\ref{E9}) is equivalent to 
\begin{equation}\label{E10}
\|\theta\|^{-2}_{L^{\infty}(\tilde{X})}\geq \frac{C}{4}\|N_{J}\|^{2}_{L^{\infty}(X)},
\end{equation}
\begin{remark}
(1) A K\"{a}hler hyperbolic manifold is a special symplectic hyperbolic manifold since the almost complex structure $J$ is integrable, i.e, the Nijenhuis tensor is zero.\\
(2) Let $(X,J,\w)$  be a closed almost K\"{a}hler manifold with sectional curvature bounded from above by a negative constant, i.e., $sec\leq -K$ for some $K>0$. We denote by $(\tilde{X},\tilde{J},\tilde{\w})$ the universal covering space of $(X,J,\w)$. Since $\pi$ is local isometry, the sectional curvature of $\tilde{X}$ also bounded from above by the negative constant $-K$. By \cite[Lemma 3.2]{CY}, there exists $1$-form $\theta$ on $\tilde{X}$ such that  
$$\tilde{\w}=d\theta\ and\ \|\theta\|_{L^{\infty}(\tilde{X})}\leq \sqrt{n}K^{-\frac{1}{2}}.$$
If the sectional curvature $K$ of $X$ large enough to ensure that
$$K\geq nC\max_{x\in X}(\tau^{\ast}-\tau)(x),$$
then $X$ is a special symplectic  hyperbolic manifold. 
\end{remark} 
At first, we can give a  lower bound on the spectra of the Laplace operator $\De_{d'}:=\frac{1}{4}(\De_{d}+\De_{d^{\La}})$ on $L^{2}$ $k$-forms $\Om^{k}_{(2)}(X)$ for $k\neq n$. The main idea is we use the identity $[\De_{d}+\De_{d^{\La}},L]=0$.
\begin{theorem}\label{T4}
Let $(X,J,\w)$ be a complete 2n-dimensional almost K\"{a}her manifold with a $d$(bounded) symplectic form $\w$ i.e., there exists a bounded $1$-form $\theta$ such that $\w=d\theta$. Then every $L^{2}$ $k$-form $\a$ on $X$ of degree $k\neq n$ satisfies the inequality 
$$\langle(\De_{d}+\De_{d^{\La}})\a,\a\rangle\geq\la^{2}_{0}\langle\a,\a\rangle,$$
where $\la_{0}$ is a strictly positive constant which depends only on $n$ and the bounded of $\theta$,
$$\la_{0}\geq const_{n}\|\theta\|^{-1}_{L^{\infty}(X)}.$$
In particular, 
$$\mathcal{H}^{k}_{(2);d'}(X)=0,$$
unless $k\neq n$.
\end{theorem}
We then extend Gromov's idea to the special symplectic hyperbolic manifold case. We can prove the Hopf conjecture in our case.
\begin{theorem}\label{T1}(=Theorem \ref{T8})
Let $(X,J,\w)$ be a closed  $2n$-dimensional special symplectic hyperbolic manifold. Then the Euler number of $X$ satisfies $$(-1)^{n}\chi(X)>0.$$
\end{theorem}
As a corollary of Theorem\ref{T1}, we have the following result.
\begin{corollary}
Let $(X,J,\w)$ be a closed 2n-dimensional almost K\"{a}hler manifold. If the curvature of $X$ satisfies 
	$$sec\leq -K<0 \ and\  K\geq nC\max_{x\in X}(\tau^{\ast}-\tau)(x),$$
	then  the Euler number of $X$ satisfies
	$$(-1)^{n}\chi(X)>0,$$
where $C$ is an uniformly positive constant only depend on $n$. 
\end{corollary}
Suppose that the Nijenhuis tensor $N_{J}$ on a complete almost K\"{a}hler manifold $(X,J,\w)$ with a $d$(bounded) symplectic form satisfies $C\|N_{J}\|_{L^{\infty}(X)}\leq\|\theta\|^{-1}_{L^{\infty}(X)}$.  We then give a lower bound on the spectra of the Laplace operator $\De_{d}:=dd^{\ast}+d^{\ast}d$ on the space of $L^{2}$ $k$-forms $\Om^{k}_{(2)}(X)$ for $k\neq n$ (see Theorem \ref{T3}). For the non-vanishing result, we use the Gromov's idea in \cite{Gro}. Therefore we can prove the Singer conjecture under the special symplectic hyperbolic case.
\begin{theorem}\label{T2}
Let $(X,J,\w)$ be a closed $2n$-dimensional special symplectic hyperbolic manifold, $\pi:(\tilde{X},\tilde{J},\tilde{\w})\rightarrow (X,J,\w)$ the universal covering map for $X$. Then the spaces of $L^{2}$-harmonic $k$-forms on its universal covering space $\tilde{X}$ satisfy
	\begin{equation*}
	\left\{
	\begin{aligned}
	\mathcal{H}^{k}_{(2)}(\tilde{X})=\{0\}, & k\neq n\\
	\mathcal{H}^{n}_{(2)}(\tilde{X})\neq\{0\}, &  \\
	\end{aligned}  
	\right.
	\end{equation*}
	is equivalent to
	\begin{equation*}
	\left\{
	\begin{aligned}
	h^{k}_{(2)}(X)=0, & k\neq n\\
	h^{n}_{(2)}(X)>0.& \\
	\end{aligned}  
	\right.
	\end{equation*}
	In particular, $$(-1)^{n}\chi(X)>0.$$
\end{theorem}

\section{Preliminaries}
\subsection{Differential forms on almost K\"{a}hler manifold }
We recall some definitions and results on the differential forms for almost complex and almost Hermitian manifolds. Let $X$ be a $2n$-dimensional manifold (without boundary) and $J$ be a smooth almost complex structure on $X$. There is a natural action of $J$ on the space $\Om^{k}(X,\C):=\Om^{k}(X)\otimes\C$, which induces a topological type decomposition
$$\Om^{k}(X,\C)=\bigoplus_{p+q=k}\Om^{p,q}_{J}(X,\C),$$
where  $\Om^{p,q}_{J}(X,\C)$ denotes the space of complex forms of type $(p,q)$ with respect to $J$. We have that
$$d:\Om^{p,q}_{J}\rightarrow\Om^{p+2,q-1}_{J}\oplus\Om^{p+1,q}_{J}\oplus\Om^{p,q+1}_{J}\oplus\Om^{p-1,q+2}_{J}$$
and so the operator $d$ splits according as
$$d=A_{J}+\pa_{J}+\bar{\pa}_{J}+\bar{A}_{J},$$
where all the pieces are graded algebra derivations, $A_{J}$, $\bar{A}_{J}$ are  $0$-order differential operators. Note that each component of $d$ is a derivation, with bi-degrees given by 
$$|A_{J}|=(2,-1),\ |\pa_{J}|=(1,0),\ |\bar{\pa}_{J}|=(0,1),\ |\bar{A}_{J}|=(-1,2).$$
The integrability theorem of Newlander and Nirenberg states that the almost complex structure $J$ is integrable if and only if $N_{J}=0$, where
$$N_{J}:TX\otimes TX\rightarrow TX,$$
denotes the Nijenhuis tensor
$$N_{J}(X,Y):=[X,Y]+J[X,JY]+J[JX,Y]-[JX,JY].$$ 
For $\a\in\Om^{1}(X)$ we have
$$(A_{J}(\a)+\bar{A}_{J}(\a))(X,Y)=\frac{1}{4}\a(N_{J}(X,Y)),$$
In particular,  $J$ is integrable if only if $N_{J}=0$, i.e, $A_{J}=0$. Expanding the equation $d^{2}=0$ we obtain the following set of equations:
\begin{equation*}
\begin{split}
&A_{J}^{2}=0,\\
&A_{J}\pa_{J}+\pa_{J}A_{J}=0,\\
&\pa^{2}_{J}+A_{J}\bar{\pa}_{J}+\bar{\pa}_{J}A_{J}=0,\\
&\pa_{J}\bar{\pa}_{J}+\bar{\pa}_{J}\pa_{J}+A_{J}\bar{A}_{J}+\bar{A}_{J}A_{J}=0,\\
&\bar{\pa}^{2}_{J}+\bar{A}_{J}\pa_{J}+\pa_{J}\bar{A}_{J}=0,\\
&\bar{A}_{J}\bar{\pa}_{J}+\bar{\pa}_{J}\bar{A}_{J}=0.\\
&\bar{A}_{J}^{2}=0.\\
\end{split}
\end{equation*} 
\begin{definition}
An almost K\"{a}hler structure on $2n$-dimensional manifold $X$ is a pair
$(\w,J)$ where $\w$ is a symplectic form and $J$ is an almost complex structure calibrated by $\w$.
\end{definition}
If $(X,\w,J)$ is an almost K\"{a}hler manifold, then
$$g(X,Y)=\w(X,JY)$$
is a $J$-Hermitian metric, i.e., $g(JX,JY)=g(X,Y)$ for any $X,Y$.  For any almost K\"{a}hler manifold $(X,J,\w)$ there is an associated Hodge-star operator \cite{Huy}
$\ast:\Om^{p,q}_{J}\rightarrow\Om^{n-q,n-p}_{J}$ defined by 
$$\a\wedge\ast\bar{\b}=\langle\a,\b\rangle\frac{\w^{n}}{n!}.$$
The Lefschetz operator  $L:\Om^{p,q}_{J}\rightarrow\Om_{J}^{p+1,q+1}$ defined by
$$L(\a)=\w\wedge\a.$$
It has adjoint $\La=\ast^{-1}L\ast$. There is a Lefschetz decomposition on complex $k$-forms
$$\Om^{k}(X,\C)=\bigoplus_{r\geq0}L^{r}P_{\C}^{k-2r},$$
where $P_{\C}^{\bullet}=\ker{\La}\cap\Om^{\bullet}(X,\C)$. 

We recall some definitions on almost Hermitian manifold \cite{MOS,Sek}. We denote by $\na$, $R$, $\rho$, $\tau$ the Levi-Civita connection, the curvature tensor, the Ricci tensor and the scalar curvature of $X$, respectively. Here, we assume that the curvature $R$ is defined by $$R(X,Y)Z=[\na_{X},\na_{Y}]Z-\na_{[X,Y]}Z$$ 
for $X,Y,Z\in\mathfrak{X}(X)$. We denote by $\{X_{1},\cdots,X_{2n}\}$ a local orthonormal frame field of $X$. We have 
$R_{ijkl}=g(R(X_{i},X_{j})X_{k},X_{l})$, $J_{ij}=g(JX_{i},X_{j})$ and $\na_{i}J_{jk}=g((\na_{X_{i}}J)X_{j},X_{k})$. We denote by $\rho^{\ast}$ and $\tau^{\ast}$ the Ricci $\ast$-tensor and $\ast$-scalar curvature defined respectively by
$$\rho^{\ast}(X,Y)=g(Q^{\ast}X,Y)=trace(Z\mapsto R(X,JZ)JX),$$
$$\tau^{\ast}=traceQ^{\ast},$$
where $X,Y,Z\in\mathfrak{X}(X)$. By using the first Bianchi identity,
we have
$$\rho^{\ast}(X,Y)=-\frac{1}{2}\sum_{i=1}^{2n}R(X,JY,X_{i},JX_{i}),$$
and hence
$$\tau^{\ast}=-\frac{1}{2}\sum_{a,b,i,j=1}^{2n}J_{ab}R_{abij}J_{ij}.$$
The following equality is known 
\begin{proposition}(\cite[Lemma 2.4 ]{Sek} or \cite[Equation (3.1)--(3.3)]{SV})\label{P6}
Let $(X,J,\w)$ be a closed almost K\"{a}hler manifold. Then
$$|\na{J}|^{2}=2(\tau^{\ast}-\tau).$$
$$(\na_{X}J)Y+(\na_{JX}J)JY=0,$$ 
$$g(N_{J}(X,Y),Z)=2g(J(\na_{Z}J)X,Y),$$
for $X,Y,Z\in\mathfrak{X}(X)$
\end{proposition}
\begin{example}
Let $(X,J,\w)$ be a closed almost K\"{a}hler manifold. We denote by $\{Z_{i}\}$ the orthonormal basis of $T_{p}^{1,0}$, $p\in X$. We extend the curvature operator $R: \wedge^{2}T_{p}X\rightarrow\wedge^{2}T_{p}X$ to a complex linear transformation $R^{\C}: \wedge^{2}T_{p}X\otimes\C\rightarrow\wedge^{2}T_{p}X\otimes\C$ \cite{Her91}. Given a nonzero decomposition $\Pi\in\La^{2}T_{p}X\otimes\C$, its complex sectional curvature is the real number
$$K^{\C}(\Pi)=\frac{(R^{\C}(\Pi),\overline{\Pi})}{(\Pi,\overline{\Pi})}$$
(one also can see \cite[Definition 2.2]{Her91}). Then following \cite[Lemma 3.3]{Her00}, we have
\begin{equation*}
\tau^{\ast}-\tau=-4\sum_{i,j=1}^{n}(R^{\C}(Z_{i}\wedge Z_{j}),\overline{Z_{i}\wedge Z_{j}}).
\end{equation*}
If the curvature of $X$ satisfies 
$$K^{\C}\geq -\frac{K}{4n^{3}C},$$
then $$Cn(\tau^{\ast}-\tau)\leq K.$$
It implies that $X$ is a special symplectic hyperbolic. 
\end{example}		
\subsection{Almost K\"{a}hler identities}
In \cite{CW1,CW2}, the authors extended the K\"{a}hler identities to the non-integrable setting and deduced several geometric and topologies consequences. In this section, we will recall almost K\"{a}hler identities constructed by Cirici-Wilson \cite{CW1}.

The operators $\de=A_{J}, \pa_{J}, \bar{\pa}_{J}, \bar{A}_{J}$ and $\de$ have $L^{2}$-adjoint operators $\de^{\ast}$ when $X$ is closed, and one can check that
$$\bar{A}^{\ast}_{J}=-\ast A_{J}\ast\quad and\quad \bar{\pa}^{\ast}_{J}=-\ast\pa_{J}\ast.$$
For any almost K\"{a}hler manifold, there is a $\mathbb{Z}_{2}$-graded Lie algebra of operators action on the $(p,q)$-forms, generated by eight odd operators
$$\bar{\pa}_{J},\pa_{J},\bar{A}_{J},A_{J},\bar{\pa}^{\ast}_{J},\pa^{\ast}_{J},\bar{A}^{\ast}_{J},A^{\ast}_{J}$$
and even degree operators $L,\La,H$ \cite[Section 3]{CW1}.

Cirici-Wilson \cite{CW1,CW2} defined the $\de$-Laplacian by letting
$$\De_{\de}:=\de\de^{\ast}+\de^{\ast}\de.$$
For all $p,q$, we will denote by
$$\mathcal{H}^{p,q}_{\de}:=\ker{\De_{\de}}\cap\Om^{p,q}_{J}=\ker{\de}\cap\ker{\de^{\ast}}\cap\Om^{p,q}_{J}$$
the space of $\de$-harmonic forms in bi-degree $(p,q)$.

By introducing a symplectic Hodge star operator $\ast_{s}$,  Brylinski \cite{Bry} proposed a Hodge theory of closed symplectic manifolds. The space of symplectic harmonic $k$-forms is $\mathcal{H}^{k}_{sym}:=\ker{d}\cap\ker{d^{\La}}\cap\Om^{k}$, where $$d^{\La}:=[d,\La]=(-1)^{k+1}\ast_{s}d\ast_{s}.$$ 
Given any compatible triple $(X,J,\w)$ on a symplectic manifold, the differential operator $d^{\La}=[d,\La]$ and the $d^{c}$ operator
$$d^{c}:=J^{-1}dJ,$$
are related via the Hodge star operator defined with respect to the compatible metric $g$ by the relation
$$d^{\La}=d^{c\ast}:=-\ast d^{c}\ast.$$
Brylinski also showed that in almost K\"{a}hler manifold, a form of pure $(p,q)$ is in $\mathcal{H}^{p+q}_{sym}$ if only if it is in $\mathcal{H}^{p+q}_{d}$. This gives an inclusion $\bigoplus_{p+q=k}\mathcal{H}^{p,q}_{d}\hookrightarrow\mathcal{H}^{p+q}_{sym}$. In general, this is strict. Indeed, Yan \cite{Yan} showed that $k=0,1,2$, every cohomology class has a symplectic harmonic representative.

We define the graded commutator of operators $A$ and $B$ by
$$[A,B]=AB-(-1)^{deg(A)deg(B)}BA$$
where $deg(A)$ denotes the total degree of $A$. 
\begin{proposition}(\cite[Proposition 3.1]{CW1})\label{P2}
For any almost K\"{a}hler manifold the following identities hold:\\
(1) $[L,\bar{A}_{J}]=[L,A_{J}]=0$ and $[\La,\bar{A}^{\ast}_{J}]=[\La,A^{\ast}_{J}]=0$.\\
(2)	$[L,\bar{\pa}_{J}]=[L,\pa_{J}]=0$ and $[\La,\bar{\pa}^{\ast}_{J}]=[\La,\pa^{\ast}_{J}]=0$.\\
(3) $[L,\bar{A}^{\ast}_{J}]=\sqrt{-1}A_{J}$, $[L,A^{\ast}_{J}]=-\sqrt{-1}\bar{A}_{J}$ and   $[\La,\bar{A}_{J}]=\sqrt{-1}A^{\ast}_{J}$, $[\La,A_{J}]=-\sqrt{-1}\bar{A}^{\ast}_{J}$.\\
(4) $[L,\bar{\pa}^{\ast}_{J}]=-\sqrt{-1}\pa_{J}$, $[L,\pa^{\ast}_{J}]=\sqrt{-1}\bar{\pa}_{J}$ and   $[\La,\bar{\pa}_{J}]=-\sqrt{-1}\pa^{\ast}_{J}$, $[\La,\pa_{J}]=\sqrt{-1}\bar{\pa}^{\ast}_{J}$.
\end{proposition}
If $C$ is another operator of degree $deg(C)$, the following Jacobi identity is easy to check
$$(-1)^{deg(C)deg(A)}\big{[}A,[B,C]\big{]}+(-1)^{deg(A)deg(B)}\big{[}B,[C,A]\big{]}+(-1)^{deg(B)deg(C)}\big{[}C,[A,B]\big{]}=0.$$
\begin{proposition}(\cite[Proposition 3.3]{CW1})\label{P3}
For any almost K\"{a}hler manifold the following identities hold:\\
	(1) $[\bar{A}_{J},A^{\ast}_{J}]=[A_{J},\bar{A}^{\ast}_{J}]=0$.\\
	(2)	$[\bar{A}_{J},\pa^{\ast}_{J}]=[\bar{\pa}_{J},A^{\ast}_{J}]$ and $[A_{J},\bar{\pa}^{\ast}_{J}]=[\pa_{J},\bar{A}^{\ast}_{J}]$.\\
	(3) $[\pa_{J},\bar{\pa}^{\ast}_{J}]=[\bar{A}^{\ast}_{J},\bar{\pa}_{J}]+[A_{J},\pa^{\ast}_{J}]$ and $[\bar{\pa}_{J},\pa^{\ast}_{J}]=[A^{\ast}_{J},\pa_{J}]+[\bar{A}_{J},\bar{\pa}^{\ast}_{J}]$.
\end{proposition}
In \cite{CW1}, they also gave several relations concerning various Laplacians.
\begin{proposition}(\cite[Proposition 3.4]{CW1})\label{P7}
For any almost K\"{a}hler manifold the following identities hold:\\
(1) $\De_{\bar{A}_{J}+A_{J}}=\De_{\bar{A}_{J}}+\De_{A_{J}}$.\\
(2)	$\De_{\bar{\pa}_{J}}+\De_{A_{J}}=\De_{\pa_{J}}+\De_{\bar{A}_{J}}.$\\
(3) $\De_{d}=2(\De_{\bar{\pa}_{J}}+\De_{A_{J}}+[\bar{A}_{J},\pa_{J}^{\ast}]+[A_{J},\bar{\pa}_{J}^{\ast}]+[\pa_{J},\bar{\pa}_{J}^{\ast}]+[\bar{\pa}_{J},\pa_{J}^{\ast}])$.
\end{proposition}
\section{Harmonic forms on symplectic manifold}
\subsection{$\De_{d'}$-harmonic forms}
In this section, we discuss the harmonic forms that can be constructed
from the two differential operators $d'$ and $d''$. We begin with the
known cohomologies with $d$ (for de Rham $H_{d}$) and $d^{\La}$ (for $H_{d^{\La}}$).

We define the operators $d'=\pa_{J}+\bar{A}_{J}$ and $d''=\bar{\pa}_{J}+A_{J}$. Those have adjoint $d'^{\ast}=\pa^{\ast}_{J}+\bar{A}^{\ast}_{J}$ and  $d''^{\ast}=\bar{\pa}^{\ast}_{J}+A^{\ast}_{J}$. Hence
$$d=d'+d'',\ d^{\ast}=d'^{\ast}+d''^{\ast}$$
and
$$d^{\La}=[d,\La]=\sqrt{-1}d'^{\ast}-\sqrt{-1}d''^{\ast},$$
$$d^{\La_{\ast}}=([d,\La])^{\ast}=[L,d^{\ast}]=-\sqrt{-1}d'+\sqrt{-1}d''.$$
The operators $d'$ and $d''$ were also previously introduced by
de Bartolomeis and Tomassini \cite{BT}. They also established some almost K\"{a}hler identities which obtained by Cirici-Wilson.

With the exterior derivative $d$, there is of course the de Rham cohomology
$$H^{k}_{dR}(X)=\frac{\ker d\cap\Om^{k}(X)}{{\rm{Im}}d\cap\Om^{k}(X)},$$
that is present on all Riemannian manifolds. Since $d^{\La}d^{\La}=0$, there is also a natural cohomology
$$H^{k}_{d^{\La}}(X)=\frac{\ker d^{\La}\cap\Om^{k}(X)}{{\rm{Im}}d^{\La}\cap\Om^{k}(X)}.$$
This cohomology has been discussed in \cite{Mat,TY1,TY2,TY3,Yan}.

We utilize the compatible triple $(X,J,g)$ on $X$ to write the Laplacian associated with the $d^{\La}$-cohomology:
$$\De_{d^{\La}}=[d^{\La},d^{\La_{\ast}}].$$
The self-adjoint Laplacian naturally defines a harmonic form. 

We consider an self-adjoint Laplacian $$\De_{d}+\De_{d^{\La}}=[d,d^{\ast}]+[d^{\La},d^{\La_{\ast}}].$$ 
Suppose that $J$ is integrable, i.e., $N_{J}=0$, then 
$$\De_{d}=\De_{d^{\La}}=2\De_{\pa}=2\De_{\bar{\pa}}.$$
In symplectic case, we introduce two self-adjoint operators $$\De_{d'}:=[d',d'^{\ast}]\ (resp.\  \De_{d''}:=[d'',d''^{\ast}]).$$ 
By the inner product
$$\langle \a,\De_{d^{\bullet}}\a\rangle_{L^{2}(X)}=\|d^{\bullet}\a\|^{2}+\|d^{\bullet\ast}\a\|^{2}$$
we are led to the following definition.
\begin{definition}
A differential form $\a\in\Om^{\ast}(X)$ is call $d^{\bullet}$-harmonic if $\De_{d^{\bullet}}\a=0$, or equivalently, $d^{\bullet}\a=d^{\bullet\ast}\a=0$. We denote the space of $d^{\bullet}$-harmonic $k$-forms by $\mathcal{H}_{d^{\bullet}}^{k}(X)$.
\end{definition}
We will show that
$$\ker\De_{d'}\cap\Om^{k}(X)=\ker\De_{d''}\cap\Om^{k}(X)=\ker(\De_{d}+\De_{d^{\La}})\cap\Om^{k}(X).$$
\begin{proposition}\label{P1}
For any almost K\"{a}hler manifold the following identities hold:\\
(1)\ $\De_{d}=[d,d^{\ast}]=[d',d'^{\ast}]+[d',d''^{\ast}]+[d'',d'^{\ast}]+[d'',d''^{\ast}].$\\
(2)\ $\De_{d^{\La}}=[d^{\La},d^{\La_{\ast}}]=[d',d'^{\ast}]-[d',d''^{\ast}]-[d'',d'^{\ast}]+[d'',d''^{\ast}].$
\end{proposition}
In \cite[Lemma 3.8]{BT}, the authors gave a relation between $\ker\De_{d}$ and $\ker\De_{d^{\La}}$ as follows.	If $\a\in\Om^{k}(X)$ is a $\De_{d}$-harmonic form in a closed almost K\"{a}hler manifold $(X,J,\w)$, then 
\begin{equation*}
\|d^{\La}\a\|^{2}+\|d^{\La_{\ast}}\a\|^{2}=4\textrm{Re}(J^{-1}(A_{J}+\bar{A}_{J})J\a,d^{\La_{\ast}}\a)+4\textrm{Re}(J^{-1}(A^{\ast}_{J}+\bar{A}^{\ast}_{J})J\a,d^{\La}\a).
\end{equation*}
\begin{lemma}\label{L4}(c.f. \cite[Lemma 3.6]{BT})
Let $(X,J,\w)$ be a closed almost K\"{a}hler manifold. Then
$$\De_{d'}=\De_{d''}=\frac{1}{4}(\De_{d}+\De_{d^{\La}}).$$
In particular,
\begin{equation*}
\ker{\De_{d}}\cap\ker{\De_{d^{\La}}}\cap\Om^{k}(X)=\ker d'\cap\ker d'^{\ast}\cap\Om^{k}(X)=\ker d''\cap\ker d''^{\ast}\cap\Om^{k}(X).
\end{equation*}
\end{lemma}
\begin{proof}
Following Proposition \ref{P1}, we get 
$$\De_{d}+\De_{d^{\La}}=2[d',d'^{\ast}]+2[d'',d''^{\ast}]=2\De_{d'}+2\De_{d''}.$$
 Noting that
\begin{equation*}
\begin{split}
[d',d'^{\ast}]&=[\pa_{J},\pa^{\ast}_{J}]+[\pa_{J},\bar{A}^{\ast}_{J}]+[\bar{A}_{J},\pa^{\ast}_{J}]+[\bar{A}_{J},\bar{A}^{\ast}_{J}]\\
&=\De_{\pa_{J}}+\De_{\bar{A}_{J}}+[\pa_{J},\bar{A}^{\ast}_{J}]+[\bar{A}_{J},\pa^{\ast}_{J}]\\
&=\De_{\bar{\pa}_{J}}+\De_{A_{J}}+[A_{J},\bar{\pa}^{\ast}_{J}]+[\bar{\pa}_{J},A^{\ast}_{J}]\\
&=[\bar{\pa}_{J}+A_{J},\bar{\pa}_{J}^{\ast}+A^{\ast}_{J}]\\
&=[d'',d''^{\ast}].\\
\end{split}
\end{equation*}
Here we use the identities $\De_{\pa_{J}}+\De_{\bar{A}_{J}}=\De_{\bar{\pa}_{J}}+\De_{A_{J}}$, $[\pa_{J},\bar{A}^{\ast}_{J}]=[A_{J},\bar{\pa}^{\ast}_{J}]$ and $[\bar{A}_{J},\pa^{\ast}_{J}]=[\bar{\pa}_{J},A^{\ast}_{J}]$ (see Proposition \ref{P3} and \ref{P7}). Therefore, we get 
$\De_{d'}=\De_{d''}=\frac{1}{4}(\De_{d}+\De_{d^{\La}})$.
\end{proof}
\subsection{Hard Lefschetz Condition}
From Lemma \ref{L4}, we know that $\mathcal{H}^{k}_{d'}(X)=\mathcal{H}^{k}_{d''}(X)\subset\mathcal{H}^{k}_{d}(X)$ and $$\dim\mathcal{H}^{k}_{d'}(X)=\dim\mathcal{H}^{k}_{d''}(X)\leq\dim\mathcal{H}^{k}_{d}(X),$$ 
it implies that $\mathcal{H}^{k}_{d^{\bullet}}(X)$ is finite dimensional. We denote the Betti numbers of $X$ by $b_{i}(X):=\dim\mathcal{H}^{i}_{d}(X)$. The Euler number of $X$ is given by
$$\chi(X)=\sum_{i\geq 0}(-1)^{i}\dim\mathcal{H}^{i}_{d}(X)=\sum_{i\geq0}(-1)^{i}b_{i}.$$
We denote the $i$-th almost K\"{a}her Betti numbers by $b^{AK}_{i}:=\dim\mathcal{H}^{i}_{d'}(X)$. Therefore we define the almost K\"{a}her Euler number as follows
$$\chi^{AK}(X)=\sum_{i\geq0}(-1)^{i}\dim\mathcal{H}^{i}_{d'}(X)=\sum_{i\geq0}(-1)^{i}b^{AK}_{i}.$$
\begin{remark}
	Suppose the almost complex structure $J$ on $X$ is integrable, i.e., $(X,J,\w)$ is K\"{a}hlerian. Hence $A_{J}=\bar{A}_{J}=0$. In this case, $\mathcal{H}^{k}_{d'}(X)=\ker\De_{\pa}\cap\Om^{k}$ and  $\mathcal{H}^{k}_{d''}(X)=\ker\De_{\bar{\pa}}\cap\Om^{k}$. By the identity $\De_{\pa}=\De_{\bar{\pa}}=\frac{1}{2}\De_{d}$, one knows that 
	$$\mathcal{H}^{k}_{d'}(X)=\mathcal{H}^{k}_{d''}(X)=\mathcal{H}^{k}_{dR}(X).$$
	Hence the almost K\"{a}her Euler numbers $\chi^{AK}(X)$ is the Euler number of $X$.
\end{remark}
A special class of symplectic manifolds is represented by those ones satisfying the Hard Lefschetz Condition (HLC), i.e., those closed $2n$-dimensional symplectic manifolds $(X,J,\w)$ for which the map
$$L^{k}:H^{n-k}_{dR}(X)\rightarrow H^{n+k}_{dR}(X),\ \forall\ 0\leq k<n$$
are isomorphisms. In particular, a classical result states if $(X,J,\w)$ is a closed K\"{a}hler manifold, then $(X,J,\w)$ satisfies the HLC \cite{Huy}. In \cite{TW}, the authors studied the Hard Lefschetz property of $\ker\De_{d}$ on almost K\"{a}hler manifold. 
\begin{theorem}(\cite[Theorem 5.2]{TW})
Let $(X,J,\w)$ be a closed $2n$-dimensional almost K\"{a}hler manifold. Then $\mathcal{H}^{k}_{d'}(X)$ satisfies the HLC.
\end{theorem}
\begin{proof}
Notice that $4\De_{d'}=\De_{d}+\De_{d^{\La}}$. We only to need to prove that $\ker\De_{d}\cap\ker\De_{d^{\La}}$ satisfies HLC. Let us start by showing that
$$[\De_{d}+\De_{d^{\La}},L]=0.$$	
In fact, by the Jacobi identity, we have
$$[L,[d,d^{\ast}]]+[d,[d^{\ast},L]]-[d^{\ast},[L,d]]=0,$$
and
$$[L,[d^{\La_{\ast}},d^{\La}]]+[d^{\La_{\ast}},[d^{\La},L]]-[d^{\La},[L,d^{\La_{\ast}}]]=0.$$
Since $[L,d]=[L,d^{\La_{\ast}}]=0$, $[d^{\ast},L]=-d^{\La_{\ast}}$, $[d^{\La},L]=d$, and
$$[d,d^{\La_{\ast}}]=[d^{\La_{\ast}},d]=d^{\La_{\ast}}d+dd^{\La_{\ast}},$$
we have
$$[L,[d,d^{\ast}]]+[L,[d^{\La_{\ast}},d^{\La}]]=0.$$
which gives $[\De_{d}+\De_{d^{\La}},L]=0$. Thus $L$ maps $\ker{\De_{d}}\cap\ker{\De_{d^{\La}}}=\ker{\De_{d}+\De_{d^{\La}}}$ to itself. A similar argument gives $[\De_{d}+\De_{d^{\La}},\La]=0$, it implies that  $\La(\ker{\De_{d}}\cap\ker{\De_{d^{\La}}})\subset \ker{\De_{d}}\cap\ker{\De_{d^{\La}}}$. Thus there is an $sl_{2}$-structure on $\ker{\De_{d}}\cap\ker{\De_{d^{\La}}}$ and our theorem follows.
\end{proof}
\begin{remark}
The above proof gives $[L,\De_{d}]=[d,d^{\La_{\ast}}]$. Since $[d,d^{\La_{\ast}}]=0$ if and only if $J$ is integrable. We then know that $X$ is K\"{a}hlerian if and only if $[L,\De_{d}]=0$.
\end{remark}
\begin{corollary}(c.f. \cite[Theorem 5.3]{TW})\label{C1}
Let $(X,J,\w)$ be a closed  $2n$-dimensional almost K\"{a}hler manifold. If
$\mathcal{H}^{k}_{d}(X)=\mathcal{H}^{k}_{d^{\La}}(X)$, for all $0\leq k<n$, then HLC on $(\mathcal{H}^{\bullet}_{d},L)$ and HLC on $(\mathcal{H}^{\bullet}_{d^{\La}},L)$.
\end{corollary}
It's easy to see that the HLC on $(\mathcal{H}^{\bullet}_{d},L)$ implies the HLC on $(H^{\bullet}_{dR},L)$. But in general, we don't know whether they are equivalent or not.  In \cite{TWZ}, the authors studied the symplectic cohomologies and symplectic harmonic forms which introduced by Tseng-Yau. Based on this, they get if $(X,J,\w)$ is a closed symplectic parabolic manifold which satisfies the hard Lefschetz property, then its Euler number satisfies the inequality $(-1)^{n}\chi(X)\geq0$. We then have
\begin{corollary}(c.f. \cite{TWZ})
Let $(X,J,\w)$ be a closed  $2n$-dimensional symplectic parabolic manifold. Suppose that $\dim\mathcal{H}^{k}_{d}(X)=\dim\mathcal{H}^{k}_{d^{\La}}(X)$ for all $0\leq k<n$. Then the Euler number of $X$ satisfies
$$(-1)^{n}\chi(X)\geq0.$$
\end{corollary}
\subsection{Index of a family elliptic operators}
Let $X$ be a closed Riemannian manifold of real dimension $\dim X$. Then for each $0\leq k\leq\dim X$, we have the following de Rham elliptic operator $\mathcal{D}_{dR}=d+d^{\ast}$:
$$\mathcal{D}_{dR}:\bigoplus_{k=even}\Om^{k}(X)\rightarrow\bigoplus_{k=odd}\Om^{k}(X),$$
whose index is the Euler number of $X$.  In fact,
\begin{equation*}
\begin{split}
Index(\mathcal{D}_{dR})&=\dim\ker{\mathcal{D}_{dR}}-\dim({\rm{coker}}{\mathcal{D}_{dR}})\\
&=\dim\bigoplus_{k=even}\mathcal{H}^{k}_{d}(X)-\dim\bigoplus_{k=odd}\mathcal{H}^{k}_{d}(X)\\
&=\sum_{k=0}^{\dim X}(-1)^{k}b_{k}\\
&=\chi(X).\\
\end{split}
\end{equation*}
For the almost K\"{a}hler manifold $(X,J,\w)$, we construct a family elliptic operator 
\begin{equation}\label{E8}
\mathcal{D}(t)=\sqrt{\frac{2}{1+t^{2}}}\big{(}(d'+td'')+(d'^{\ast}+td''^{\ast})\big{)}
:\bigoplus_{k=even}\Om^{k}(X)\rightarrow\bigoplus_{k=odd}\Om^{k}(X).
\end{equation}
Hence $\mathcal{D}(0)=d'+d'^{\ast}$ and $\mathcal{D}(1)=\mathcal{D}_{dR}$.
\begin{proposition}\label{P4}(\cite[Proposition 3.3]{TY1})
For any $t\in\mathbb{R}^{\geq0}$, 
$$\mathcal{D}^{2}(t):\bigoplus_{k=even/odd}\Om^{k}(X)\rightarrow\bigoplus_{k=even/odd}\Om^{k}(X)$$
 is an elliptic differential operator.
\end{proposition}
\begin{proof}
	To calculate the symbol of $\mathcal{D}^{2}(t)$, we will work a local unitary frame of $T^{\ast}X$ and choose a basis $\{e^{1},\cdots,e^{n}\}$ such that the metric is written as
	$$g=e^{i}\otimes\bar{e}^{i}+\bar{e}^{i}\otimes e^{i},$$
	with $i=1,\cdots,n$. With an almost
	complex structure $J$, any $k$-form can be decomposed into a sum of $(p,q)$-forms with $p+q=k$. We can write a $(p,q)$-form in the local moving-frame coordinates
	$$A_{p,q}=A_{i_{1}\cdots i_{p}j_{1}\cdots j_{q}}e^{i_{1}}\wedge\cdots e^{i_{p}}\wedge \bar{e}^{j_{1}}\wedge\cdots\bar{e}^{j_{p}}.$$
	The exterior derivative then acts as
	\begin{equation}\label{E7}
	dA_{p,q}=(\pa A_{p,q})_{p+1,q}+(\bar{\pa}A_{p,q})_{p,q+1}+A_{i_{1}\cdots i_{p}j_{1}\cdots j_{q}}d(e^{i_{1}}\wedge\cdots e^{i_{p}}\wedge \bar{e}^{j_{1}}\wedge\cdots\bar{e}^{j_{p}}),
	\end{equation}
	where
	$$(\pa A_{p,q})_{p+1,q}=\pa_{i_{p+1}}A_{i_{1}\cdots i_{p}j_{1}\cdots j_{q}}e^{i_{1}}\wedge\cdots e^{i_{p}}\wedge \bar{e}^{j_{1}}\wedge\cdots\bar{e}^{j_{p}}$$
	$$(\bar{\pa}A_{p,q})_{p,q+1}=\bar{\pa}_{j_{q+1}}A_{i_{1}\cdots i_{p}j_{1}\cdots j_{q}}e^{i_{1}}\wedge\cdots e^{i_{p}}\wedge \bar{e}^{j_{1}}\wedge\cdots\bar{e}^{j_{p}}.$$
	In calculating the symbol, we are only interested in the highest-order
	differential acting on $A_{i_{1}\cdots i_{p}j_{1}\cdots j_{q}}$. Therefore, only the first two terms of (\ref{E7}) are relevant for the calculation. In dropping the last term, we are effectively working in $\mathbb{C}^{n}$ and can make use of all the K\"{a}hler identities involving derivative operators. So, effectively, we have (using $\simeq$ to denote equivalence under symbol calculation)
	$$d'+td''\simeq \pa_{J}+t\bar{\pa}_{J},$$
	$$d'^{\ast}+td''^{\ast}\simeq\pa^{\ast}_{J}+t\bar{\pa}^{\ast}_{J}.$$
	Noting that the highest-order of the operators $[\pa_{J},\pa^{\ast}_{J}]$,  $[\bar{\pa}_{J},\bar{\pa}^{\ast}_{J}]$ differential acting on $A_{i_{1}\cdots i_{p}j_{1}\cdots j_{q}}$ are the same, since $\De_{\pa_{J}}+\De_{\bar{A}_{J}}=\De_{\bar{\pa}}+\De_{A_{J}}$. We thus have
	\begin{equation*}
	\begin{split}
	\mathcal{D}^{2}(t)&\simeq\frac{2}{1+t^{2}}\big{(} [\pa_{J},\pa^{\ast}_{J}]+t^{2}[\bar{\pa}_{J},\bar{\pa}^{\ast}_{J}]\big{)}\\
	&\simeq 2[\pa_{J},\pa^{\ast}_{J}]\\
	&\simeq ([\pa_{J},\pa^{\ast}_{J}]+[\bar{\pa}_{J},\bar{\pa}^{\ast}_{J}])\\
	&\simeq \De_{d}.
	\end{split}
	\end{equation*}
\end{proof}
\begin{corollary}\label{C2}
	Let $(X,J,\w)$ be a closed $2n$-dimensional almost K\"{a}hler manifold. Then 
	$$\chi(X)=Index(\mathcal{D}(0)).$$
\end{corollary}
\begin{proof}
	The operator $\mathcal{D}(t)$ is self-adjoint. Following Proposition \ref{P4}, we know that $\mathcal{D}^{2}(t)$ is a generalized Laplacian. Hence $\mathcal{D}(t)$ is a Dirac type operator in the sense of \cite[Definition 2.1.17]{Nic}. Naturally, the operator $\mathcal{D}(t)$ is elliptic. By Theorem  \cite[Theorem 2.1.32]{Nic}, for any $t\in[0,1]$, we have
	$$Index(\mathcal{D}(t))=Index(\mathcal{D}(1)).$$
	Noticing that $Index(\mathcal{D}(1))=\chi(X)$. Thus we have $\chi(X)=Index(\mathcal{D}(0))$.
\end{proof}
We recall the Atiyah's $L^{2}$ index \cite{Ati,Pan}.
\begin{theorem}\cite[Theorem 6.1]{Pan}
	Let $X$ be a closed Riemannian manifold, $P$ a determined elliptic operator on sections of certain bundles over $X$. Denote  by $\tilde{\mathcal{D}}$ its lift of $\mathcal{D}$ to the universal convering space $\tilde{X}$. Let $\Ga=\pi_{1}(M)$. Then   the $L^{2}$ kernel of $\tilde{P}$ has a finite $\Ga$-dimension and 
	$$L^{2}Index_{\Ga}(\tilde{P})=Index(P).$$
\end{theorem}
We denote by $\tilde{\mathcal{D}}(t)$ the lifted elliptic operator of $\mathcal{D}(t)$. We then have
\begin{corollary}\label{C4}
	Let $(X,J,\w)$ be a closed $2n$-dimensional almost K\"{a}hler manifold, $\pi:(\tilde{X},\tilde{J},\tilde{\w})\rightarrow (X,J,\w)$ the universal covering maps for $X$. Let $\Ga=\pi_{1}(X)$. Then the Euler number of $X$ satisfies 
	$$\chi(X)=Index(\mathcal{D}(0))= L^{2}Index_{\Ga}(\tilde{\mathcal{D}}(0))$$
\end{corollary}
\begin{remark}
Noticing that 
$$d'^{2}=\pa_{J}^{2}+[\pa_{J},\bar{A}_{J}]=\pa^{2}_{J}-\bar{\pa}^{2}_{J}.$$
The operator $d'^{2}$ always not zero. But it's easy to see 
$$\bigoplus_{k=even/odd}\mathcal{H}^{k}_{(2);d'}\subset\ker(d'+d'^{\ast})\cap\bigoplus_{k=even/odd}\Om^{k}_{(2)}.$$
\end{remark}

\section{Euler number of the hyperbolic symplectic manifolds}
\subsection{Vanishing theorems}
We begin the proof the Theorem \ref{T4} by recalling some basis notions in Hodge theory and almost K\"{a}hler geometry. If $X$ is an oriented complete Riemannian manifold, let $d^{\ast}$ be the adjoint operator of $d$ acting on the space of $L^{2}$ $k$-forms. Denote by $\Om^{k}_{(2)}(X)$ and $\mathcal{H}^{k}_{(2)}(X)$ the spaces of $L^{2}$ $k$-forms and $L^{2}$ harmonic $k$-forms, respectively. By elliptic regularity and completeness of the manifold, a $k$-form in $\mathcal{H}^{k}_{(2)}(X)$ is smooth, closed and co-closed.

Suoppse that $(X,J,\w)$ is a complete almost K\"{a}hler manifold. We denote by
$$\mathcal{H}^{k}_{(2);d'}(X)=\{\a\in\Om^{k}_{(2)}(X):\De_{d'}\a=0 \}$$
the space of $L^{2}$ $\De_{d'}$-harmonic $k$-forms on $X$.
	
We follow the method of Gromov's \cite{Gro} to choose a sequence of cutoff functions $\{f_{\varepsilon}\}$ satisfying the following conditions:\\
(i) $f_{\varepsilon}$ is smooth and takes values in the interval $[0,1]$, furthermore, $f_{\varepsilon}$ has compact support.\\
(ii) The subsets $f^{-1}_{\varepsilon}\subset X$, i.e., of the points $x\in X$ where $f_{\varepsilon}(x)=1$ exhaust $X$ as $\varepsilon\rightarrow0$.\\
(iii) The differential of $f_{\varepsilon}$ everywhere bounded by $\varepsilon$,
$$\|df_{\varepsilon}\|_{L^{\infty}(X)}=\sup_{x\in X}|df_{\varepsilon}|\leq\varepsilon.$$
Thus one obtains another useful
\begin{lemma}
Let $(X,J,\w)$ be a complete almost K\"{a}hler manifold. If an $L^{2}$ $k$-form $\a$ is $\De_{d'}$-harmonic, then $d'\a=d'^{\ast}\a=0$.
\end{lemma}
\begin{proof}
	We want to justify the integral identity
	$$\langle\De_{d'}\a,\a\rangle=\langle d'\a,d'\a\rangle+\langle d'^{\ast}\a,d'^{\ast}\a\rangle$$
	If $d'\a$ and $d'^{\ast}\a$  are $L^{2}$ (i.e., square integrable on $X$), then this follows by Lemma \cite[1.1. A]{Gro}. To handle the general case we cutoff $\a$ and obtain by a simple computation
	\begin{equation*}
	\begin{split}
	0&=\langle\De_{d'}\a,f^{2}_{\varepsilon}\a\rangle\\
	&=\langle d'\a,d'(f^{2}_{\varepsilon}\a)\rangle+\langle d'^{\ast}\a,d'^{\ast}(f^{2}_{\varepsilon}\a)\rangle\\
	&=\langle d'\a,f^{2}_{\varepsilon}(d'\a)\rangle+\langle d'\a,2f_{\varepsilon}\pa f_{\varepsilon}\wedge\a\rangle+\langle d'^{\ast}\a,f^{2}_{\varepsilon}(d'^{\ast}\a)\rangle-\langle d'^{\ast}\a,\ast(2f_{\varepsilon}\bar{\pa} f_{\varepsilon}\wedge\ast\a)\rangle\\
	&=I_{1}(\varepsilon)+I_{2}(\varepsilon),\\
	\end{split}
	\end{equation*}
	where
	\begin{equation*}
	\begin{split}
	|I_{1}(\varepsilon)|&=\langle d'\a,f^{2}_{\varepsilon}d'\a\rangle+\langle d'^{\ast}\a,f^{2}_{\varepsilon}d'^{\ast}\a\rangle\\
	&=\int_{X}f^{2}_{\varepsilon}(|d'\a|^{2}+|d'^{\ast}\a|^{2})\\
	\end{split}
	\end{equation*}
	and
	\begin{equation*}
	\begin{split}
	|I_{2}(\varepsilon)|&=|\langle d'\a,2f_{\varepsilon}\pa f_{\varepsilon}\wedge\a\rangle-\langle d'^{\ast}\a,\ast(2f_{\varepsilon}\bar{\pa} f_{\varepsilon}\wedge\ast\a)\rangle|\\
	&\leq|\langle d'\a,2f_{\varepsilon}\pa f_{\varepsilon}\wedge\a\rangle|+|\langle d'^{\ast}\a,\ast(2f_{\varepsilon}\bar{\pa} f_{\varepsilon}\wedge\ast\a)\rangle|\\
	&\lesssim\int_{X}|df_{\varepsilon}|\cdot|f_{\varepsilon}|\cdot|\a|(|d'\a|+|d'^{\ast}\a|).\\
	\end{split}
	\end{equation*}
	Then we choose $f_{\varepsilon}$ such that $|df_{\varepsilon}|^{2}<\varepsilon f_{\varepsilon}$ on $X$ and estimate $I_{2}$ by Schwartz inequality. Then
	$$|I_{2}(\varepsilon)|\lesssim \varepsilon\|f_{\varepsilon}\a\|_{L^{2}(X)}\big{(}\int_{X}f^{2}_{\varepsilon}(|d'\a|^{2}+|d'^{\ast}\a|^{2})\big{)}^{\frac{1}{2}},$$
	and hence $|I_{2}|\rightarrow0$ for $\varepsilon\rightarrow0$.
\end{proof}
Now, we begin to give the lower bound on the spectrum of the operator $\De_{d'}$ on $\Om^{k}_{(2)}$ for $k\neq n$.
\begin{proof}[\textbf{Proof of Theorem \ref{T4}}]
	To simply notation we shall write $a\lesssim b$ for $a\leq const_{n}b$ and $a\approx b$, for $b\lesssim a\lesssim b$. Then we recall the operator $L^{k}:\Om^{p}\rightarrow\Om^{2n-p}$ for a given $p<n$ and $p+k=n$. By the Lefschetz theorem $L^{k}$ is a bijective quasi-isometry and so every $L^{2}$-form $\psi$ of degree $2n-p$ is the product $\psi=L^{k}\phi=\w^{k}\wedge\phi$, where $\phi\in\Om^{p}_{(2)}$ and 
	$$\|\psi\|_{L^{2}(X)}\approx\|\phi\|_{L^{2}(X)}.$$
	Since $L^{k}$ commutes with $\De_{d}+\De_{d^{\La}}$, we also have
	\begin{equation*}
	\begin{split}
	\langle(\De_{d}+\De_{d^{\La}})\psi,\psi\rangle&=\langle L^{k}(\De_{d}+\De_{d^{\La}})\phi,L^{k}\phi\rangle\\
	&\approx\langle(\De_{d}+\De_{d^{\La}})\phi,\phi\rangle.
	\end{split}
	\end{equation*}
	Then we write $\psi=d\eta+\psi'$, for $\eta=\theta\wedge\w^{k-1}\wedge\phi$ and $\psi'=\theta\wedge\w^{k-1}\wedge d\phi$, and observe that
	\begin{equation*}
	\begin{split}
	\|\eta\|_{L^{2}(X)}&\lesssim\|\theta\|_{L^{\infty}(X)}\|\phi\|_{L^{2}(X)}\\
	&\lesssim\|\theta\|_{L^{\infty}(X)}\|\psi\|_{L^{2}(X)}.
	\end{split}
	\end{equation*}
	Next, since
	$$\|d\phi\|^{2}_{L^{2}(X)}\lesssim\langle\De_{d}\phi,\phi\rangle\lesssim\langle(\De_{d}+\De_{d^{\La}})\phi,\phi\rangle\lesssim\langle(\De_{d}+\De_{d^{\La}})\psi,\psi\rangle,$$
	we have $$\|\psi'\|_{L^{2}(X)}\lesssim\|\theta\|_{L^{\infty}(X)}\langle(\De_{d}+\De_{d^{\La}})\psi,\psi\rangle^{\frac{1}{2}}.$$
	Now,
	$$\|\psi\|^{2}_{L^{2}(X)}=\langle\psi,\psi\rangle=\langle\psi,d\eta+\psi'\rangle\lesssim|\langle\psi,d\eta\rangle|+|\langle\psi,\psi'\rangle|,$$
	where 
	\begin{equation*}
	\begin{split}
	|\langle\psi,d\eta\rangle|&=|\langle d^{\ast}\psi,\eta\rangle|\\
	&\leq \langle d^{\ast}\psi,d^{\ast}\psi\rangle^{\frac{1}{2}}\|\eta\|_{L^{2}(X)}\\
	&\leq \langle \De_{d}\psi,\psi\rangle^{\frac{1}{2}}\|\eta\|_{L^{2}(X)}\\
	&\lesssim \langle\De_{d}\psi,\psi\rangle^{\frac{1}{2}}\|\theta\|_{L^{\infty}(X)}\|\psi\|_{L^{2}(X)}\\
	&\lesssim\|\theta\|_{L^{\infty}(X)}\langle(\De_{d}+\De_{d^{\La}})\psi,\psi\rangle^{\frac{1}{2}}\|\psi\|_{L^{2}(X)}
	\end{split}
	\end{equation*}
	and 
	$$|\langle\psi,\psi'\rangle|\leq\|\psi\|_{L^{2}(X)}\|\psi'\|_{L^{2}(X)}\lesssim\|\theta\|_{L^{\infty}(X)}\langle(\De_{d}+\De_{d^{\La}})\psi,\psi\rangle^{\frac{1}{2}}\|\psi\|_{L^{2}(X)}.$$
	This yields the desired estimate
	$$\|\phi\|_{L^{2}(X)}\lesssim\|\psi\|_{L^{2}(X)}\lesssim\langle(\De_{d}+\De_{d^{\La}})\psi,\psi\rangle^{\frac{1}{2}}\lesssim\langle(\De_{d}+\De_{d^{\La}})\phi,\phi\rangle^{\frac{1}{2}}$$
	for the forms $\phi$ of degree $p<n$. The case $p>n$ follows by the Poincar\'{e} duality as the operator $\ast:\Om^{p}\rightarrow\Om^{2n-p}$ commutes with $\De_{d}+\De_{d^{\La}}$ and is isometric for the $L^{2}$-norms.	
\end{proof}
For the $d$(sublinear) case, we prove the following result.
\begin{proposition}\label{P5}
	Let $(X,J,\w)$ be a complete 2n-dimensional almost K\"{a}hler manifold with a $d$(sublinear) symplectic form $\w$. Then for any $k\neq n$, 
	$$\mathcal{H}^{k}_{(2);d'}(X)=\{0\}.$$
\end{proposition}
\begin{proof}
	By hypothesis, there exists a 1-form $\theta$ with $\w=d\theta$ and
	$$\|\theta(x)\|_{L^{\infty}(X)}\leq c(1+\rho(x,x_{0})),$$
	where $c$ is an absolute constant. In what follows we assume that the distance function $\rho(x, x_{0})$ is smooth for $x\neq x_{0}$. The general case follows easily by an approximation argument.
	
	Let $\eta:\mathbb{R}\rightarrow\mathbb{R}$ be smooth, $0\leq\eta\leq1$,
	$$
	\eta(t)=\left\{
	\begin{aligned}
	1, &  & t\leq0 \\
	0,  &  & t\geq1
	\end{aligned}
	\right.
	$$
	and consider the compactly supported function
	$$f_{j}(x)=\eta(\rho(x_{0},x)-j),$$
	where $j$ is a positive integer.
	
	Let $\a$ be a $(\De_{d}+\De_{d^{\La}})$-harmonic $k$-form in $L^{2}$, $k<n$, and consider the form $\Phi=\a\wedge\theta$. Observing that $d^{\ast}(\a\wedge\w)=d^{\La_{\ast}}(\a\wedge\w)=0$ since $\w\wedge\a$ is a $(\De_{d}+\De_{d^{\La}})$-harmonic $(k+2)$-form in $L^{2}$, and noticing  that $f_{j}\Phi$ has compact support, one has
	\begin{equation}\label{E1}
	0=\langle d^{\ast}(\w\wedge\a),f_{j}\Phi\rangle=\langle\w\wedge\a,d(f_{j}\Phi)\rangle.
	\end{equation}
	We further note that, since $\w=d\theta$ and $d\a=0$,
	\begin{equation}\label{E2}
	\begin{split}
	0&=\langle\w\wedge\a,d(f_{j}\Phi)\rangle\\
	&=\langle\w\wedge\a,f_{j}d\Phi\rangle+\langle\w\wedge\a,df_{j}\wedge\Phi\rangle\\
	&=\langle\w\wedge\a,f_{j}\w\wedge\a\rangle+\langle\w\wedge\a,df_{j}\wedge\theta\wedge\a\rangle.\\
	\end{split}
	\end{equation}
	Since $0\leq f_{j}\leq 1$ and $\lim_{j\rightarrow\infty}f_{j}(x)(\ast\a)(x)=\ast\a(x)$, it follows from the dominated convergence theorem that
	\begin{equation}\label{E3}
	\lim_{j\rightarrow\infty}\langle \w\wedge\a,f_{j}\w\wedge\a\rangle_{L^{2}(X)}=\|\w\wedge\a\|^{2}_{L^{2}(X)}.
	\end{equation}
	Since $\w$ is bounded, $supp(df_{j})\subset B_{j+1}\backslash B_{j}$ and $\|\theta(x)\|_{L^{\infty}}=O(\rho(x_{0},x))$, one obtains that
	\begin{equation}\label{E4}
	|\langle\w\wedge\a,df_{j}\wedge\theta\wedge\a\rangle|\leq (j+1)C\int_{B_{j+1}\backslash B_{j}}|\a(x)|^{2}dx,
	\end{equation}
	where $C$ is a constant independent of $j$.
	
	We claim that there exists a subsequence $\{j_{i}\}_{i\geq1}$ such that
	\begin{equation}\label{E5}
	\lim_{i\rightarrow\infty}(j_{i}+1)\int_{B_{j_{i}+1}\backslash B_{j_{i}}}|\a(x)|^{2}dx=0.
	\end{equation}
	If not, there exists a positive constant $a$ such that
	$$\lim_{j\rightarrow\infty}(j+1)\int_{B_{j+1}\backslash B_{j}}|\a(x)|^{2}dx\geq a>0.$$
	This inequality implies
	\begin{equation}\nonumber
	\begin{split}
	\int_{X}|\a(x)|^{2}dx&=\sum_{j=0}^{\infty}\int_{B_{j+1}\backslash B_{j}}|\a(x)|^{2}dx\\
	&\geq a\sum_{j=0}^{\infty}\frac{1}{j+1}\\
	&=+\infty\\
	\end{split}
	\end{equation}
	which is a contradiction to the assumption $\int_{X}|\a(x)|^{2}dx<\infty$. Hence, there exists a subsequence $\{j_{i}\}_{i\geq1}$ for which (\ref{E5}) holds. Using (\ref{E4}) and (\ref{E5}), one obtains
	\begin{equation}\label{E6}
	\lim_{j\rightarrow\infty}\langle\w\wedge\a,df_{j}\wedge\theta\wedge\a\rangle=0
	\end{equation}
	It now follows from (\ref{E2}), (\ref{E3}) and (\ref{E6}) that $\w\wedge\a=0$. Since $L$ is injective $k<n$, $\a=0$ as desired.
\end{proof}
Following Theorem \ref{T4}, we then have
\begin{theorem}\label{T5}
	Let $(X,J,\w)$ be a complete $2n$-dimensional almost K\"{a}hler manifold with $d$(bounded) symplectic form $\w$, i.e., there exists a bounded $1$-form $\theta$ such that $\w=d\theta$. There is a uniform positive constant $C$ only depends on $n$ with following significance. If the Nijenhuis tensor $N_{J}$ satisfies
	$$C\|N_{J}\|_{L^{\infty}(X)}\leq\|\theta\|^{-1}_{L^{\infty}(X)},$$
	then when $n=odd/even$,
	$$\ker(d'+d'^{\ast})\cap\bigoplus_{k=even/odd}\Om^{k}_{(2)}=\{0\}.$$
\end{theorem}
\begin{proof}
	We only prove the case of $n=odd$. Let $\a=\sum_{k=0}^{n}\a_{2k}\in\Om^{even}(X)$ be an $L^{2}$-form on $X$, where $\a_{2k}\in\Om^{2k}_{(2)}(X)$. Noting that
	\begin{equation*}
	\begin{split}
	(d'+d'^{\ast})^{2}&=[d',d'^{\ast}]+(d')^{2}+(d'^{\ast})^{2}\\
	&=[d',d'^{\ast}]+(\pa_{J}^{2}+[\pa_{J},\bar{A}_{J}])+((\pa^{\ast}_{J})^{2}+[\pa^{\ast}_{J},\bar{A}^{\ast}_{J}])\\
	&=\frac{1}{4}(\De_{d}+\De_{d^{\La}})+([\pa_{J},\bar{A}_{J}]-[\bar{\pa}_{J},A_{J}])+([\pa^{\ast}_{J},\bar{A}^{\ast}_{J}]-[\bar{\pa}^{\ast}_{J},A^{\ast}_{J}]).\\
	\end{split}
	\end{equation*}
	Here we use the identities $\pa_{J}^{2}+[\bar{\pa}_{J},A_{J}]=0$ and $(\pa^{\ast}_{J})^{2}+[\bar{\pa}^{\ast}_{J},A^{\ast}_{J}]=0$.
	By the inner product
	\begin{equation*}
	\begin{split}
	I:&=\langle([\pa_{J},\bar{A}_{J}]-[\bar{\pa}_{J},A_{J}]+[\pa^{\ast}_{J},\bar{A}^{\ast}_{J}]-[\bar{\pa}^{\ast}_{J},A^{\ast}_{J}])\a,\a\rangle\\
	&=\langle([\pa_{J}+\bar{A}_{J},\bar{A}_{J}]-[\bar{\pa}_{J}+A_{J},A_{J}]+[\pa^{\ast}_{J}+\bar{A}^{\ast}_{J},\bar{A}^{\ast}_{J}]-[\bar{\pa}^{\ast}_{J}+A^{\ast}_{J},A^{\ast}_{J}])\a,\a\rangle\\
	&=\langle([d',\bar{A}_{J}]-[d'',A_{J}]+[d'^{\ast},\bar{A}^{\ast}_{J}]-[d''^{\ast},A^{\ast}_{J}])\a,\a\rangle\\
	&=2{\rm{Re}}(\langle d'^{\ast}\a,\bar{A}_{J}\a\rangle+\langle d'\a,\bar{A}^{\ast}_{J}\a\rangle-\langle A_{J}\a,d''^{\ast}\a\rangle-\langle d''^{\ast}\a,A^{\ast}_{J}\a\rangle).\\
	\end{split}
	\end{equation*}
	Here we use the identities $A^{2}_{J}=\bar{A}^{2}_{J}=(A^{\ast}_{J})^{2}=(\bar{A}^{\ast}_{J})^{2}=0$. Therefore, we get
	\begin{equation*}
	\begin{split}
	|I|&\leq 2(\|d'^{\ast}\a\|\cdot\|\bar{A}_{J}\a\|+\| d'\a\|\cdot\|\bar{A}^{\ast}_{J}\a\|+\| A_{J}\a\|\cdot\|d''^{\ast}\a\|+\|d''^{\ast}\a\|\cdot\|A^{\ast}_{J}\a\|)\\
	&\leq C(\|d'^{\ast}\a\|+\| d'\a\|+\|d''\a\|+\|d''^{\ast}\a\|)\|N_{J}\|_{L^{\infty}(X)}\|\a\|\\
	&\leq C\varepsilon(\|d'^{\ast}\a\|^{2}+\| d'\a\|^{2}+\|d''\a\|^{2}+\|d''^{\ast}\a\|^{2})+\frac{1}{2\varepsilon}\|N_{J}\|^{2}_{L^{\infty}(X)}\|\a\|^{2}\\
	&\leq \frac{ C\varepsilon}{2}\langle(\De_{d}+\De_{d^{\La}})\a,\a\rangle+\frac{1}{2\varepsilon}\|N_{J}\|^{2}_{L^{\infty}(X)}\|\a\|^{2},
	\end{split}
	\end{equation*}
	where $C$ is a positive constant and where we have used the inequality
	$$2ab\leq\varepsilon a^{2}+\frac{1}{\varepsilon}b^{2},$$ 
	for any $\varepsilon>0$ and any real numbers $a$ and $b$. Noting that 
	\begin{equation*}
	\begin{split}
	\langle(\De_{d}+\De_{d^{\La}})\a,\a\rangle&=\langle\sum_{k=0}^{n}(\De_{d}+\De_{d^{\La}})\a_{2k},\a_{2k}\rangle\\
&=\sum_{k=0}^{n}\langle(\De_{d}+\De_{d^{\La}})\a_{2k},\a_{2k}\rangle.\\
	\end{split}
	\end{equation*}
	Following Theorem \ref{T4}, we then have
	\begin{equation*}
	\langle(\De_{d}+\De_{d^{\La}})\a,\a\rangle\geq\sum_{k=0}^{n}\la_{0}^{2}\langle\a_{2k},\a_{2k}\rangle=\la_{0}^{2}\|\a\|^{2}.
	\end{equation*}
	 Therefore, we get
	\begin{equation*}
	\begin{split}
	\langle (d'+d'^{\ast})^{2}\a,\a\rangle&\geq\frac{1}{4}\langle(\De_{d}+\De_{d^{\La}})\a,\a\rangle-|I|\\
	&\geq(\frac{1}{4}-\frac{1}{2}C\varepsilon)\langle(\De_{d}+\De_{d^{\La}})\a,\a\rangle-\frac{1}{2\varepsilon}\|N_{J}\|^{2}_{L^{\infty}(X)}\|\a\|^{2}\\
	&\geq((\frac{1}{4}-\frac{1}{2}C\varepsilon)\la^{2}_{0}-\frac{1}{2\varepsilon}\|N_{J}\|^{2}_{L^{\infty}(X)})\|\a\|^{2},\\
	\end{split}
	\end{equation*}
 We take $\varepsilon=\frac{1}{4C}$ and $2C\|N_{J}\|^{2}_{L^{\infty}(X)}\leq\frac{\la^{2}_{0}}{16}$, hence
	\begin{equation*}
	\langle (d'+d'^{\ast})^{2}\a,\a\rangle\geq\frac{\la^{2}_{0}}{16}\|\a\|^{2}.
	\end{equation*}
	Therefore, $\ker(d'+d'^{\ast})\cap\bigoplus_{k= even}\Om^{k}_{(2)}=\{0\}$. 
\end{proof}

\subsection{Non-vanishing theorems}
Let $E$ and $E'$ be $C^{\infty}$-vector bundles over a smooth manifold $X$, and $\mathcal{D}:C^{\infty}(E)\rightarrow C^{\infty}(E')$ be a differential operator between $C^{\infty}$-sections of these bundle. We also suppose that $X$ is a Riemannian manifold and $\Ga$ is a discrete group of isometrics of $X$, such that the differential operator $\mathcal{D}$ commutes with the action of $\Ga$. We consider a $\Ga$-invariant Hermitian line bundle $(L,\na)$ on $X$ we assume $X/\Ga$ is compact, and we state Atiyah's $L^{2}$-index theorem for $\mathcal{D}\otimes\na$.
\begin{theorem}\cite[Theorem 2.3.A]{Gro}\label{T7}
Let $\mathcal{D}$ be a first-order elliptic operator. Then there exists a closed nonhomogeneous form
$$I_{D}=I^{0}+I^{1}+\cdots+I^{n}\in\Om^{\ast}(X)=\Om^{0}\oplus\Om^{1}\oplus\cdots\oplus\Om^{n},\ n=\dim X,$$
invariant under $\Ga$, such that the $L^{2}$-index of the twisted operator $\mathcal{D}\otimes\na$ satisfies
$$L^{2}Index_{\Ga}(\mathcal{D}\otimes\na)=\int_{X/\Ga}I_{\mathcal{D}}\wedge\exp{[\w]},$$
where $[\w]$ is the Chern form of $\na$, and
$$\exp{[\w]}=1+[\w]+\frac{[\w]\wedge[\w]}{2!}+\frac{[\w]\wedge[\w]\wedge[\w]}{3!}+\cdots.$$
\end{theorem}
\begin{remark}
(1)  $L^{2}Index_{\Ga}(\mathcal{D}\otimes\na)\neq 0$ implies that
either $\mathcal{D}\otimes\na$ or its adjoint has a non-trivial $L^{2}$-kernel.\\
(2) The operators $\mathcal{D}$ used in the present paper are the operators $d+d^{\ast}$ and $d'+d'^{\ast}$. In these cases the $I_{0}$-component of $I_{\mathcal{D}}$ is non-zero. Hence $\int_{X/\Ga}I_{\mathcal{D}}\wedge\exp{\a[\w]}\neq0$, for almost all $\a$, provided the curvature form $[\w]$ is ``homologically nonsingular" $\int_{X/\Ga}[\w]^{n}\neq0$.
\end{remark}
Gromov defined the lower spectral bound $\la_{0}=\la_{0}(\mathcal{D})\geq 0$ as the upper bound of the negative numbers $\la$, such that $\|\mathcal{D}e\|_{L^{2}}\geq\la\|e\|_{L^{2}}$ for those sections $e$ of $E$ where $\mathcal{D}e$ in $L^{2}$. Let $\mathcal{D}$ be a $\Ga$-invariant elliptic operator on $X$ of the first order, and let $I_{D}=I^{0}+I^{1}+\cdots+I^{n}\in\Om^{\ast}(X)$ be the corresponding index form on $X$. Let $\w$ be a closed $\Ga$-invariant $2$-form on $X$ and denote by $I_{\a}^{n}$ the top component of product $I_{\mathcal{D}}\wedge\exp{\a\w}$, for $\a\in\mathbb{R}$. Hence $I_{\a}^{n}$ is an $\Ga$-invariant $n$-form on $X$, $\dim X=n$ depending on parameter $\a$.
\begin{theorem}(\cite[2.4.A. Theorem]{Gro})\label{T6}
Let $H^{1}_{dR}(X)=0$ and let $X/\Ga$ be compact and $\int_{X/\Ga}I_{\a}^{n}\neq 0$, for some $\a\in\mathbb{R}$. If the form $\w$ is $d$(bounded), then either $\la_{0}(\mathcal{D})=0$ or $\la_{0}(\mathcal{D}^{\ast})=0$, where $\mathcal{D}^{\ast}$ is the adjoint operator.
\end{theorem}
Let $X$ be a closed almost K\"{a}hler manifold, with exact symplectic form $\w=d\theta$ on $\tilde{X}$. Let $\Ga=\pi_{1}(X)$. For each $\varepsilon$, $\na_{\varepsilon}=d+\sqrt{-1}\varepsilon\theta$ is a unitary connection on the trivial line bundle $L=\tilde{X}\times\mathbb{C}$. One can try made it $\Ga$-invariant by changing to a non-trivial action of $\Ga$ on $\tilde{X}\times\mathbb{C}$, i.e., setting, for $\gamma\in\Ga$,
$$\gamma^{\ast}(\tilde{x},z)=(\gamma\tilde{x},\exp^{\sqrt{-1}u(\gamma,\tilde{x})}z).$$
We want $\gamma^{\ast}\na_{\varepsilon}=\na_{\varepsilon}$, i.e., $du=-(\gamma^{\ast}\theta-\theta)$. Since $d(\gamma^{\ast}\theta-\theta)=\gamma^{\ast}\w-\w=0$, there always exists a solution $u(\gamma,\cdot)$, well defined up a constant. 

However, one cannot adjust the constant $\varepsilon\w$ to obtain an action (if so, one would get a line bundle on $X$ with curvature $\varepsilon\w$ and first Chern class $\frac{\varepsilon}{2\pi}[\w]$). This means that the action only defined on a central extension, we call this projective representation (see \cite[Charp 9]{Pan}).
\begin{definition}(\cite[Definition 9.2]{Pan})
	Let $G_{\varepsilon}$ be the subgroup of $Diff(\tilde{X}\times\mathbb{C})$ formed by maps $g$ which are linear unitary on fibers, preserve the connection $\na_{\varepsilon}$ and cover an element of $\Gamma$.	
\end{definition}
By construction we have an exact sequence
$$1\rightarrow U(1)\rightarrow G_{\varepsilon}\rightarrow\Gamma\rightarrow 1.$$
Since sections of the line bundle $\tilde{X}\times\mathbb{\C}\rightarrow\tilde{X}$ can be viewed as $U(1)$ equivalent functions on $\tilde{X}\times U(1)$, the operator $$\tilde{P}_{\varepsilon}=\frac{\sqrt{2}}{2}(d^{\na_{\varepsilon}}+\sqrt{-1}[L,(d^{\ast})^{\na_{\varepsilon}}])+\frac{\sqrt{2}}{2}((d^{\ast})^{\na_{\varepsilon}}-\sqrt{-1}[d^{\na_{\varepsilon}},\La])$$ can be view as a $G_{\varepsilon}$ invariant operator on the Hilbert space $H$ of $U(1)$ equivariant basis $L^{2}$ differential forms on $\tilde{X}\times U(1)$.  Following Theorem \ref{T7}, the $L^{2}$-index of the operator $\tilde{P}_{\varepsilon}$ satisfies
$$L^{2}Index_{G_{\varepsilon}}(\tilde{P}_{\varepsilon})=\int_{X}I_{\tilde{P}_{\varepsilon}}\wedge\exp(\frac{\varepsilon}{2\pi}[\w]),$$
where $I_{\tilde{P}}:=I^{0}+I^{1}+\cdots$ denotes the Atiyah-Bott-Patodi index form of $\tilde{P}:=\sqrt{2}(d'+d'^{\ast})$. Therefore $L^{2}Index_{G_{\varepsilon}}(\tilde{P}_{\varepsilon})$ is a polynomial in $\varepsilon$ whose highest degree term is $\int_{X}(\frac{\w}{2\pi})^{n}\neq 0$.
\begin{theorem}\label{T8}
Let $(X,J,\w)$ be a closed $2n$-dimensional special symplectic hyperbolic  manifold, $\pi:(\tilde{X},\tilde{J},\tilde{\w})\rightarrow (X,J,\w)$ the universal covering maps for $X$. Let $\Ga=\pi_{1}(X)$. Then when $n=even/odd$, we have
\begin{equation*}
\left\{
\begin{aligned}
\ker{\tilde{P}}\cap\bigoplus_{k=odd/even}\Om^{k}_{(2)}(\tilde{X})=\{0\}, \\
\ker{\tilde{P}}\cap\bigoplus_{k=even/odd}\Om^{k}_{(2)}(\tilde{X})\neq\{0\}.
\end{aligned}  
\right.
\end{equation*} 
In particular,	$$(-1)^{n}\chi(X)>0.$$
\end{theorem}
\begin{proof}
The universal covering space $\tilde{X}$ is simply-connected and the lifted symplectic form $\tilde{\w}$ is $\tilde{d}$(bounded).  For $\varepsilon$ small enough, by Theorem \ref{T6}, either $\ker\tilde{P}\cap\oplus_{k=even}\Om^{k}_{(2)}\neq{0}$ or $\ker\tilde{P}\cap\oplus_{k=odd}\Om^{k}_{(2)}\neq{0}$ since $\int_{X}[\w]^{n}\neq 0$.  By the hypothesis $N_{J}$ is bounded  from above by  $\|\theta\|_{L^{\infty}(X)}^{-1}$, we get $\pi^{\ast}(N_{J})$ is also bounded from above by $\|\theta\|_{L^{\infty}(\tilde{X})}^{-1}$ since $X$ is closed and $\pi$ is a local isometry.  When $n=even/odd$, according to Theorem \ref{T5}, the spectrum of $\tilde{P}$ lies way apart from zero from possible  the forms which belong to $\ker{\tilde{P}}\cap\oplus_{k= even/odd}\Om^{k}_{(2)}$. Therefore, $\ker\tilde{P}\cap\oplus_{k= even/odd}\Om^{k}_{(2)}\neq{0}$. By the Positivity of the Von-Neumann Dimension \cite[Section 2.1]{Pan}, we get 
$$\dim_{\Ga}(\ker{ \tilde{P}}\cap\bigoplus_{k= even/odd}\Om^{k}_{(2)})>0,\quad \dim_{\Ga}(\ker{ \tilde{P}}\cap\bigoplus_{k=odd/even}\Om^{k}_{(2)})=0.$$
Therefore following Corollary \ref{C4}, we have
$$  (-1)^{n}\chi(X)=(-1)^{n}L^{2}Index_{\Ga}\tilde{P}=\dim_{\Ga}(\ker\tilde{P}\cap\bigoplus_{k=even/odd}\Om^{k}_{(2)})>0.$$
\end{proof}

\subsection{$L^{2}$-Hodge number}\label{S1}
We assume throughout this subsection that $(X,g,J)$ is a closed almost K\"{a}hler $2n$-dimensional manifold with a Hermitian metric $g$, and $\pi:(\tilde{X},\tilde{g},\tilde{J})\rightarrow(X,g,J)$ its universal covering with $\Gamma$ as an isometric group of deck transformations. Denote by $\mathcal{H}^{k}_{(2)}(\tilde{X})$ the spaces of $L^{2}$-harmonic $k$-forms on $\Om^{k}_{(2)}(\tilde{X})$, where $\Om^{k}_{(2)}(\tilde{X})$ is space of the squared integrable $k$-forms on $(\tilde{X},\tilde{g},\tilde{J})$, and denote by $\dim_{\Gamma}\mathcal{H}^{k}_{(2)}(\tilde{X})$ the Von Neumann dimension of $\mathcal{H}^{k}_{(2)}(\tilde{X})$ with respect to $\Gamma$ \cite{Ati,Pan}. We denote by $h_{(2)}^{k}(X)$ the $L^{2}$-Hodge numbers of $X$, which are defined to be $$h_{(2)}^{k}(X):=\dim_{\Gamma}\mathcal{H}_{(2)}^{k}(\tilde{X}),\ (0\leq k\leq 2n).$$ 
It turns out that $h^{k}_{(2)}(X)$ are independent of the Hermitian metric $g$ and depend only on $X$ and $J$. By the $L^{2}$-index theorem of Atiyah \cite{Ati}, we have the following crucial identities between $\chi(X)$ and the $L^{2}$-Hodge numbers $h_{(2)}^{k}(X)$:
$$\chi(X)=\sum_{k=0}^{2n}(-1)^{k}h_{(2)}^{k}(X).$$
Now, We give a lower bound on the spectra of the Laplace operator $\De_{d}:=dd^{\ast}+d^{\ast}d$ on $L^{2}$-forms $\Om^{k}(X)$ for $k\neq n$.
\begin{theorem}\label{T3}
	Let $(X,J,\w)$ be a complete $2n$-dimensional almost K\"{a}hler manifold with $d$(bounded) symplectic form $\w$, i.e., there exists a bounded $1$-form $\theta$ such that $\w=d\theta$. There is a uniform positive constant $C$ only depends on $n$ with following significance. If the Nijenhuis tensor $N_{J}$ satisfies
$$C\|N_{J}\|_{L^{\infty}(X)}\leq\|\theta\|^{-1}_{L^{\infty}(X)},$$
	then every $L^{2}$ $k$-form $\a$ on $X$ of degree $k\neq n$ satisfies the inequality 
	$$\langle \De_{d}\a,\a\rangle\geq \frac{1}{8}\la^{2}_{0}\langle\a,\a\rangle,$$
	where $\la_{0}$ is positive constant in Theorem \ref{T4}. In particular, $$\mathcal{H}^{k}_{d}(X)=\{0\}$$
	unless $k\neq n$.
\end{theorem}
\begin{proof}
	Following the identity on Proposition \ref{P1}, we get 
	\begin{equation*}
	\begin{split}
	\De_{d}&=\frac{1}{2}(\De_{d}+\De_{d^{\La}})+2[\pa_{J},\bar{\pa}^{\ast}_{J}]+2[\bar{\pa}_{J},\pa^{\ast}_{J}]\\
	&=\frac{1}{2}(\De_{d}+\De_{d^{\La}})+2[\bar{A}_{J}^{\ast},\bar{\pa}_{J}]+2[A_{J},\pa^{\ast}_{J}]+2[A^{\ast}_{J},\pa_{J}]+2[\bar{A}_{J},\bar{\pa}_{J}^{\ast}],\\
	&=\frac{1}{2}(\De_{d}+\De_{d^{\La}})+2(I_{1}+I_{2}+I_{3}+I_{4})\\
	\end{split}
	\end{equation*}
	Here we use the identities (3) in Proposition \ref{P3}. By the inner product
	\begin{equation*}
	\begin{split}
	\langle I_{1}\a,\a\rangle&=\langle\bar{\pa}_{J}\a,\bar{A}_{J}\a\rangle+\langle \bar{A}^{\ast}_{J}\a,\bar{\pa}^{\ast}_{J}\a\rangle\\
	&=\langle(\bar{\pa}_{J}+A_{J})\a,\bar{A}_{J}\a\rangle-\langle A_{J}\a,\bar{A}_{J}\a\rangle+\langle \bar{A}^{\ast}_{J}\a,(\bar{\pa}^{\ast}_{J}+A^{\ast}_{J})\a\rangle-\langle \bar{A}^{\ast}_{J}\a,A^{\ast}_{J}\a\rangle\\
	&=\langle d''\a,\bar{A}_{J}\a\rangle-\langle A_{J}\a,\bar{A}_{J}\a\rangle+\langle \bar{A}^{\ast}_{J}\a,d''^{\ast}\a\rangle-\langle \bar{A}^{\ast}_{J}\a,A^{\ast}_{J}\a\rangle\\
	&=\langle d''\a,\bar{A}_{J}\a\rangle+\langle \bar{A}^{\ast}_{J}\a,d''^{\ast}\a\rangle-\langle [\bar{A}^{\ast}_{J},A_{J}]\a,\a\rangle\\
	&=\langle d''\a,\bar{A}_{J}\a\rangle+\langle \bar{A}^{\ast}_{J}\a,d''^{\ast}\a\rangle.\\
	\end{split}
	\end{equation*}
Here we use the identity $[\bar{A}^{\ast}_{J},A_{J}]=0$. Similarly, 
	\begin{equation*}
	\begin{split}
	&\langle I_{2}\a,\a\rangle=\langle d'^{\ast}\a,A^{\ast}_{J}\a\rangle+\langle A_{J}\a,d'\a\rangle\\
	&\langle I_{3}\a,\a\rangle=\langle d'\a,A_{J}\a\rangle+\langle A^{\ast}_{J}\a,d'^{\ast}\a\rangle\\
	&\langle I_{4}\a,\a\rangle=\langle d''^{\ast}\a,\bar{A}^{\ast}_{J}\a\rangle+\langle \bar{A}_{J}\a,d''\a\rangle\\
	\end{split}
	\end{equation*}
	Therefore, we obtain
	\begin{equation*}
	\begin{split}
	|\langle (I_{1}+I_{4})\a,\a\rangle|&\leq 2\|d''\a\|\cdot\|\bar{A}_{J}\a\|+2\|\bar{A}^{\ast}_{J}\a\|\cdot\|d''^{\ast}\a\|\\
	&\leq \frac{1}{4}(\|d''\a\|^{2}_{L^{2}(X)}+\|d''^{\ast}\a\|^{2}_{L^{2}(X)})+4\|N_{J}\|^{2}_{L^{\infty}(X)}\|\a\|^{2}_{L^{2}(X)} \\
	&\leq\frac{1}{4} \langle[d'',d''^{\ast}]\a,\a \rangle+4\|N_{J}\|^{2}_{L^{\infty}(X)}\|\a\|^{2}_{L^{2}(X)} \\
	&\leq\frac{1}{16}\langle(\De_{d}+\De_{d^{\La}})\a,\a\rangle+4\|N_{J}\|^{2}_{L^{\infty}(X)}\|\a\|^{2}_{L^{2}(X)},\\
	\end{split}
	\end{equation*}
	and
	\begin{equation*}
	\begin{split}
	|\langle (I_{2}+I_{3})\a,\a\rangle|&\leq 2\|d'\a\|\|A_{J}\a\|+2\|A^{\ast}_{J}\a\|\|d'^{\ast}\a\|\\
	&\leq \frac{1}{4}(\|d'\a\|^{2}_{L^{2}(X)}+\|d'^{\ast}\a\|^{2}_{L^{2}(X)})+4\|N_{J}\|^{2}_{L^{\infty}(X)}\|\a\|^{2}_{L^{2}(X)} \\
	&\leq\frac{1}{4} \langle[d',d'^{\ast}]\a,\a \rangle+4\|N_{J}\|^{2}_{L^{\infty}(X)}\|\a\|^{2}_{L^{2}(X)} \\
	&\leq\frac{1}{16}\langle(\De_{d}+\De_{d^{\La}})\a,\a\rangle+4\|N_{J}\|^{2}_{L^{\infty}(X)}\|\a\|^{2}_{L^{2}(X)},\\
	\end{split}
	\end{equation*}
	where $C$ is a positive constant only depend on $n$.
	Combining above inequalities, we get
	\begin{equation*}
	\begin{split}
	\langle\De_{d}\a,\a\rangle
	&\geq\frac{1}{2}\langle(\De_{d}+\De_{d^{\La}})\a,\a\rangle-2|\langle (I_{1}+I_{4})\a,\a\rangle|-2|\langle (I_{2}+I_{3})\a,\a\rangle|\\
	&\geq \frac{1}{4}\langle(\De_{d}+\De_{d^{\La}})\a,\a\rangle-16C\|N_{J}\|^{2}_{L^{\infty}(X)}\|\a\|^{2}_{L^{2}(X)}\\
	&\geq(\frac{1}{4}\la^{2}_{0}-16C\|N_{J}\|^{2}_{L^{\infty}(X)})\|\a\|^{2}_{L^{2}(X)}.\\
	\end{split}
	\end{equation*}
If the Nijenhuis tensor $N_{J}$ satisfies  $$16C\|N_{J}\|^{2}_{L^{\infty}}\leq\frac{1}{8}\la^{2}_{0},$$ then
	$$\langle\De_{d}\a,\a\rangle\geq\frac{\la^{2}_{0}}{8}\|\a\|^{2}_{L^{2}(X)}.$$
\end{proof}
Let $L\rightarrow X$ be a vector bundle equipped with a Hermitian metric and Hermitian connection $\na$. Then there is an induced exterior differential $d^{\na}$ on $\Om^{\ast}(X)\otimes L$. If $\mathcal{D}=d^{\na}+(d^{\ast})^{\na}$, then Atiyah-Singer's index theorem states
$$Index(\mathcal{D})=\int_{X}\mathcal{L}_{X}\wedge Ch(L).$$
Here $\mathcal{L}_{X}$ is Hizebruch's class, 
$$\mathcal{L}_{X}=1+\cdots+e(X)$$
where $1\in H^{0}(X)$ and $e(X)\in H^{\dim X}(X)$ is the Euler class.  For each $\varepsilon$, $\na_{\varepsilon}=d+\sqrt{-1}\varepsilon\theta$ is a unitary connection on the trivial line bundle $L=\tilde{X}\times\mathbb{C}$. The operator $\mathcal{D}_{\varepsilon}:=d^{\na_{\varepsilon}}+(d^{\ast})^{\na_{\varepsilon}}$ can be view as a $G_{\varepsilon}$ operator on the Hilbert space $H$ of $U(1)$ equivalent basis $L^{2}$ differential forms on $\tilde{X}\times U(1)$ \cite{Pan}.
\begin{theorem}(\cite[Theorem 9.3]{Pan})
The operator $\tilde{D}_{\varepsilon}$ has a finite projective $L^{2}$ index give by
$$L^{2}Index_{G_{\varepsilon}}(\tilde{D}_{\varepsilon})=\int_{X}\mathcal{L}_{X}\wedge\exp(\frac{\varepsilon}{2\pi}[\w]).$$
\end{theorem}
Now, we begin to proof the Singer conjecture under the special symplectic hyperbolic case.
\begin{proof}[\textbf{Proof of Theorem \ref{T2}}]
This number is a polynomial in $\varepsilon$ whose highest degree term is $\int_{X}(\frac{\w}{2\pi})^{n}\neq 0$ thus for $\varepsilon$ small enough, $\tilde{D}_{\varepsilon}$ has a non-zero $L^{2}$ kernel. By construction, $\tilde{D}_{\varepsilon}$ is an $\varepsilon$-small perturbation of $\widetilde{d+d^{\ast}}$, so $\widetilde{d+d^{\ast}}$ is not invertible. For any $k\neq n$, $\mathcal{H}^{k}_{(2)}(\tilde{X})=\{0\}$, i.e., $h_{(2)}^{k}(X)=0$. Therefore, we get $\mathcal{H}^{n}_{(2)}(\tilde{X})\neq\{0\}$, i.e, $h_{(2)}^{n}(X)>0$. Hence
$$(-1)^{n}\chi(X)=(-1)^{n}\sum_{k=0}^{2n}(-1)^{k}h^{k}_{(2)}(X)=h^{n}_{(2)}(X)>0.$$
\end{proof}

\section*{Acknowledgements}
We would like to thank Professor H.Y. Wang for drawing our attention to the symplectic cohomology. I would like to thank S.O. Wilson and J. Cirici for helpful comments regarding their article \cite{CW1,CW2}. We would also like to thank the anonymous referee for careful reading of my manuscript and helpful comments. This work is supported by the National Natural Science Foundation of China (Nos. 12271496, 11801539) and the Youth Innovation Promotion Association CAS, the Fundamental Research Funds of the Central Universities, the USTC Research Funds of the Double First-Class Initiative.

\end{document}